\documentclass[12pt]{wart}
\usepackage{xspace,amssymb,amsfonts,euscript,eufrak,amsthm,amsmath}
\usepackage{graphicx,ifpdf}
\ifpdf \usepackage{epstopdf}
 \usepackage{pdfsync}\fi
\usepackage{palatino}

 \title[Standard objects in 2-braid groups]
 {Standard objects in 2-braid groups}
\author{\textsc{Nicolas Libedinsky}&\textsc{Geordie Williamson}}
\address{Departamento de Matem\'eticas & Max-Planck-Institut\\
        Universidad de Chile, Casilla 653,  & f\"ur Mathematik,\\
Las Palmeras 3425, Nu\~noa, & Vivatsgasse 7, \\
Santiago, Chile.& Bonn, Germany.}
\email{nlibedinsky@gmail.com}
\email{geordie@mpim-bonn.mpg.de}


\input xy
\xyoption {all}

  \newcommand{\nc}{\newcommand}
  \newcommand{\renc}{\renewcommand} 

\usepackage[latin1]{inputenc}

\nc{\E}{\mathbb{E}}

\nc{\Q}{\mathbb{Q}}

\nc{\triright}{\stackrel{[1]}{\to}}

\def\to{\rightarrow}

\def\longto{\longrightarrow}

\def\onto{\twoheadrightarrow}

\nc{\Br}{\mathcal{B}}
\nc{\id}{id}
\nc{\HotRR}{{}_R\mathcal{K}_R}
\nc{\HotR}{\mathcal{K}_R}
\nc{\excise}[1]{}
\nc{\defect}{\text{df}}
\nc{\h}[1]{\underline{H}_{#1}}
\nc{\Z}{\mathbb{Z}}
\nc{\R}{\mathbb{R}}
\nc{\C}{\mathbb{C}}
\renc{\P}{\mathbb{P}}
\renc{\O}{\mathcal{O}}
\nc{\N}{\mathbb{N}}

\nc{\F}{\mathcal{F}}
\nc{\G}{\mathcal{G}}

\nc{\nilp}{\mathcal{N}}

\nc{\Ga}{\mathbb{G}_a} 
\nc{\Gm}{\mathbb{G}_m} 

\nc{\Loc}{\mathcal{L}}

\nc{\A}{\mathbb{A}} 

\nc{\IC}{\mathbf{IC}}
\nc{\D}{\mathbb{D}}

\DeclareMathOperator{\Hom}{Hom}
\DeclareMathOperator{\supp}{supp}

\DeclareMathOperator{\End}{End}

  \newtheorem{defi}{Definition}
  \newtheorem{thm}{Theorem}[section]
  \newtheorem{lem}[thm]{Lemma}
  
  \newtheorem{prop}[thm]{Proposition}
  \newtheorem{cor}[thm]{Corollary}

  \newtheorem{ex}[thm]{Example}
  \theoremstyle{remark}
  \newtheorem{remark}{Remark}

\DeclareMathOperator{\cone}{cone}

\DeclareMathOperator{\Ext}{Ext}

\newcommand{\Id}{\mathrm{Id}}

\newcommand{\CO}{\mathcal{O}}

\newcommand{\Mod}[1]{{#1}\textrm{-Mod}}
\newcommand{\BMod}[2]{{#1}\textrm{-Mod-}#2}
\newcommand{\Gr}{\textrm{Gr}}

  \date{\dateline{Jan 4, 2007}{Feb 10, 2007}}

\begin{document}
\begin{abstract}
For any Coxeter system we  establish the existence (conjectured by
Rouquier) of analogues of standard and costandard objects in 2-braid
groups. This generalizes a known extension vanishing formula in the
BGG category $\mathcal{O}.$ 
\end{abstract}
\maketitle

\let\thefootnote\relax

\section{Introduction}

In \cite{Ro1} Rouquier introduces a categorification of (a quotient
of) the Artin braid group associated to a Coxeter system. He calls the
resulting monoidal category the 2-braid group. It occurs in
the study of categories of representations of semi-simple Lie
algebras, affine Lie algebras, reductive algebraic groups,
quantum groups etc. On the other hand, it has been
used to categorify the HOMFLYPT polynomial of a link (see
\cite{Kh}).

Let us be more precise. Let $(W,\mathcal{S})$ be a Coxeter system and
$V$  the geometric representation of $W$ over the complex numbers.
Let $R$ be the regular functions on $V$, graded such that
$\mathrm{deg}V^*=2$. The group $W$ acts on $V$, so by functoriality it
acts on $R$. For $s\in \mathcal{S}$ let $R^s$ be the subspace of $R$
of $s$-fixed points. Consider the following complexes of graded
$R$-bimodules: 
\[
\begin{array}{lrrl}
F_s & = & \dots \to 0\rightarrow & R\otimes_{R^s}R(1) \rightarrow R(1)
\rightarrow 0 \to \dots \\
F_{s^{-1}}& = & \quad \dots \to 0\rightarrow R(-1)\rightarrow &
R\otimes_{R^s}R(1)\rightarrow 0 \to \dots .
\end{array}
\]
Here $(1)$ denotes the grading shift functor (normalised so that
$R(1)$ is generated in degree -1). In both cases $R\otimes_{R^s}R(1)$
is the degree zero term of the complex and the non-trivial
differentials of $F_s$ and $F_{s^{-1}}$ are the unique non-zero maps of
degree zero (which are well-defined up to a scalar). 

Let $K^b(\BMod{R}{R})$ denote the homotopy category of bounded
complexes of graded $R$-bimodules, which is a monoidal category under
tensor product of complexes. Let $B_W$ be the Artin braid group
associated to the Coxeter system $(W,\mathcal{S})$. Given any word $\sigma$
in $\mathcal{S}$ and $\mathcal{S}^{-1}$ one can
consider the corresponding product of the complexes $F_s$ and
$F_{s^{-1}}$ above. In \cite{Ro1} Rouquier shows that the
corresponding complex $F_{\sigma}$  in the homotopy category only depends up to
isomorphism on the image of $\sigma$ in the Braid group $B_W$. One obtains
in this way a ``weak categorification'': a homomorphism from $B_W$ to
the set of isomorphisms classes of complexes in $K$ (which has the
structure of a monoid induced from the monoidal structure on $K$).

Moreover, Rouquier \cite{Ro1} shows that any two expressions for $\sigma \in B_W$
give rise to canonically isomorphic complexes. Hence one obtains a
``strict categorification''. That is, one has a monoidal functor
\[
F : \Omega B_W \to K^b(\BMod{R}{R})
\]
where $\Omega B_W$ is the monoidal category associated to $B_W$: the objects of
$\Omega B_W$ are the elements $\sigma \in B_W$, the morphisms
are given by $\Hom(\sigma, \sigma') = \emptyset$ if $\sigma \ne \sigma'$ and
$\End(\sigma) = \{ \id \}$, and the monoidal structure is given by the
group structure on $B_W$. 

Let $\mathfrak{B}_W$ denote the full subcategory of $K^b(\BMod{R}{R})$
consisting of all objects isomorphic to objects in the image of
$F$. Rouquier calls $\mathfrak{B}_W$ the \emph{2-braid group} of the
Coxeter system $(W,\mathcal{S})$. (Strictly speaking, what Rouquier
  calls the 2-braid group is a strict monoidal category monoidally
  equivalent to $\mathfrak{B}_W$.)
As we have seen, the decategorification of
$\mathfrak{B}_W$ is a quotient of $B_W$. He conjectures that the
decategorification is exactly $B_W$. This conjecture is true for the
topological braid group (that is when $W$ is the symmetric
group) by work of Khovanov and Seidel \cite{KhSe}.

As we have seen, when viewed as a monoidal category the morphism
spaces in $\Omega B_W$ are boring; all information is already
contained in the structure of $B_W$ as a group. This is far from true
for $\mathfrak{B}_W$. There is a rich and at present poorly understood
structure in the morphism spaces of $\mathfrak{B}_W$.

The reader seeking an analogy might like to think about the braid
group of type $A_n$, where one can view braids topologically. In this
case one has a natural categorification, the category of ``braid cobordisms'':
objects are topological braids and morphisms are 
certain cobordisms (see \cite{CS}  or the introduction of \cite{KT}). In this case
there is a monoidal functor from the category of braid cobordisms to
the 2-braid group (see \cite{KT} and \cite{EK}). Unfortunately, certain generating
cobordisms (either the birth or death of a single crossing) are
necessarily mapped to zero. So the (topological) category of braid cobordisms and
the (algebraic) 2-braid group seem to be quite different.

This paper can be seen as a first attempt to understand the
homomorphisms in $\mathfrak{B}_W$. More precisely, we explain how the canonical
section of the projection 
\[
B_W \onto W
\]
allows one to define ``standard'' and ``costandard'' objects in
$\mathfrak{B}_W$. Given $w \in W$, let $\sigma \in B_W^+$ denote the
canonical positive lift of $w$ in $B_W$ and define $F_w =
F_\sigma$. Our main theorem is the following:

\begin{thm}\label{A}\label{Princ} For $x, y \in W$ we have
\begin{equation*}
\mathrm{Hom}(F_{x},F_{y^{-1}}^{-1}[i])\cong
\begin{cases}
R& \text{if } x=y \text{ and } i=0,\\
0& \text{otherwise.}
\end{cases}
\end{equation*}
\end{thm}

Our theorem was conjectured in Rouquier's ICM address
\cite[4.2.1]{Ro2}. Hence we refer to this equation as Rouquier's
formula. This formula generalizes for all Coxeter groups the important
formula in BGG category $\mathcal{O}$ 
\begin{equation}\label{ext}
\mathrm{Ext}^i(\Delta(w),\nabla(v))=0\ \
  \mathrm{if }\ w\neq v\in W\ \mathrm{or}\ i\neq 0
\end{equation}
where $\Delta(w)$ and $\nabla(v)$ are respectively the standard and
costandard objects in (the principal block of) category $\mathcal{O}.$

In this paper we establish Theorem \ref{A} for all Coxeter groups, thus
establishing analogues of standard
and costandard objects. Let $\mathcal{B}$ denote the additive
category of Soergel bimodules (see Section \ref{SB}), and $K^b(\mathcal{B})$ denote its homotopy
category. If $W$ is a Weyl group then $K^b(\mathcal{B})$ is closely
related to both
the derived category of the principal block of category $\mathcal{O}$
(see Section \ref{2.1}) and to the 
derived category of mixed equivariant sheaves on the flag variety (see \cite{Sch}). In the
first instance the the complexes $F_x$
describe tilting resolutions of standard objects (see Lemma \ref{til}); in the
second instance they describe the subquotients in the
weight filtrations of standard sheaves (see \cite[Section 3.5]{WW}).

\begin{remark} The conjecture in \cite{Ro2} has a misprint, it is
  trivially false as stated: take for example $b=sr$ and $b'=rs$, with
  $s\neq r\in \mathcal{S}$. Theorem \ref{Princ} is the correct formulation of
  this conjecture. 
\end{remark}

\begin{remark}For $i=0$, Theorem \ref{A} follows directly from the
  construction of the light leaves basis in \cite{Li1} 
\end{remark}

As explained by Rouquier in \cite{Ro2}, proving Theorem \ref{A} should
shed light on the search for a presentation of the monoidal category of Soergel bimodules by generators and relations. Indeed unpacking Theorem
\ref{A} gives many non-trivial lifting properties of morphisms between
Soergel bimodules (this is because if $\mathcal{B}$ is the category of
Soergel bimodules as above then $\mathfrak{B}_W\subseteq K^b(\mathcal{B})$). Note however that such a generators and relations
description for the category of Soergel bimodules has recently been obtained along different lines by Elias
and the second author \cite{EW} (following work of the first author \cite{li3}, 
Elias-Khovanov \cite{EKh} and Elias \cite{E}). 

On the other hand, it is desirable to have a
generators and relations description for the monoidal category
$\mathfrak{B}_W$. At present this seems like a difficult problem,
however we hope that Theorem \ref{A} (as well as the notion of
$\Delta$ and $\nabla$-exact complexes) provide a stepping stone
towards such a description. 

 \subsubsection*{}This paper is structured as follows.  Section \ref{200}
 is intended as an introduction  to understand where Rouquier's
 formula comes from and why Theorem \ref{Princ} can be seen as a
 generalization of the Ext formula \eqref{ext} in category
 $\mathcal{O}$.  In Section \ref{1.1} we fix some general notation. In
 Section \ref{SB} we give preliminaries about
 Soergel bimodules. In Section \ref{explicit} we introduce the 2-braid
 group. Section \ref{2.1} recalls some basic facts about BGG category
 $\mathcal{O}$ and in Section \ref{1.4} we prove that Theorem
 \ref{Princ} follows for Weyl groups from \eqref{ext}. This
 last result is  a particular case of Theorem
 \ref{Princ} but we hope that this section helps the reader gain
 some understanding of where Rouquier's formula comes from.

Section \ref{sosho} is devoted to the proof of Theorem \ref{Princ}, 
so we work in the context of arbitrary Coxeter groups.
In Section \ref{ca} we recall  that Soergel bimodules are filtered by 
geometrically defined submodules. To this filtration one associate "subquotient functors"
 from the category $\mathcal{B}$ of Soergel bimodules to the category of graded $R$-bimodules.
 Soergel's Hom formula says that one can recover the Hom space between two Soergel bimodules
 by knowing the Hom spaces between the succesive subquotients of these
 filtrations (see \eqref{eq:SH}).

In Section \ref{delta} we introduce the notion of $\Delta$- and
$\nabla$-exact complexes. Roughly speaking, these are complexes which
have good exactness properties under the "subquotient functors" mentioned above.
 Using Soergel's Hom formula we  prove that the Hom space between a
 $\Delta$-exact complex and any complex of Soergel bimodules  
in the homotopy category is zero. In Section \ref{1.2} we prove the main technical result 
used to prove Theorem \ref{Princ}, namely certain augmentations of the complexes
$F_\sigma$ (resp. $F_{\sigma}^{-1}$) are $\Delta$ (resp. $\nabla$-exact) if $\sigma$ is a
positive lift of an element of $W$. This is a very strong
property. Indeed, it appears to be if and only if, though we cannot
prove this. Finally, we
use the triangulated category structure of the homotopy category in
Section \ref{2.3} to conclude the proof of Theorem \ref{Princ}.

\begin{small}
\subsection{Acknowledgements} We would like to thank Wolfgang Soergel
who explained how useful it is to know that a complex of Soergel
bimodules is $\Delta$-exact. We also benefited from discussions with
Rapha\"el Rouquier. We would like to thank Catharina Stroppel, Thorge
Jensen and the referee for useful comments. Finally, both authors would like to
thank the MPI for a productive research environment.
\end{small}

\section{Preliminaries}\label{200}

\subsection{Graded bimodules and their homotopy categories}\label{1.1}

Given graded algebras $R$, $S$ over a field $k$ we denote by $\Mod{R}$ and $\BMod{R}{S}$ the categories of $\mathbb{Z}$-graded left $R$-modules and $\mathbb{Z}$-graded $(R,S)$-bimodules respectively. (The capitalized ``M'' is intended to remind us that we are considering graded modules.)  Morphisms in $\Mod{R}$ and $\BMod{R}{S}$ are those morphisms of bimodules which preserve the grading (that is are of degree zero). We write $\hom_{\Mod{R}}$, $\hom_{\BMod{R}{S}}$ (or $\hom$ if the context is clear) for homomorphisms in these categories. Given a graded (bi)module $M = \bigoplus_{i \in \mathbb{Z}} M_i$ we define the shifted module $M(n)$ by $M(n)_i = M_{n+i}$ and set
\[\Hom(M,N) = \bigoplus_{i \in \mathbb{Z}} \hom(M,N(i))\]
to be the graded vector space of all (bi)module homomorphisms between $M$ and $N$. Let us emphasise that $\Hom(M,N)$ is only used to simplify notation at some points; it does not refer to the morphisms in any category that we consider in this paper.

Given a Laurent polynomial with positive coefficients $P = \sum a_i v^i \in \mathbb{N}[v,v^{-1}]$ and a graded (bi)module $M$ we set
\[ P \cdot M := \bigoplus M(i)^{\oplus a_i}. \]
Given $M, N \in \BMod{R}{R}$ we denote their tensor product simply by
juxtaposition: $MN := M \otimes_R N$. This tensor product
makes $\BMod{R}{R}$ into a monoidal category. 

Given an additive category $\mathcal{A}$ we denote by $K^b(\mathcal{A})$ its homotopy category of bounded complexes, which is obtained as the quotient of the category of bounded chain complexes by the ideal of null-homotopic morphisms. We use upper indices to indicate the terms of a complex, and all chain complexes will be cohomological. That is, an object $A$ of $K^b(\mathcal{A})$ is a complex of the form
\[
\dots \to A^i \to A^{i+1} \to \dots
\]where each $A^i \in \mathcal{A}$ and only finitely many $A^i$ are
non-zero. 
If $\mathcal{A}$ is in addition a monoidal category then we
obtain an induced monoidal structure on $K^b(\mathcal{A})$ given by
tensor product of complexes. Again, we denote the operation of tensor
product by juxtaposition.

Given $A, B \in K^b(\mathcal{A})$ we write $\hom_K(A,B)$ for the
homomorphisms in $K^b(\mathcal{A})$, and $\hom^\bullet(A,B)$ for
the total complex of the double complex with $(i,j)^{th}$-term
$\hom(A^i, B^j)$ and differentials induced by the differentials on $A$ and $B$ \cite[11.7]{KS}. We
have $\hom_K(A,B) = H^0(\hom^\bullet(A,B))$.

We will primarily be concerned with the homotopy category
$K^b(\BMod{R}{R})$ of graded bimodules over a graded ring $R$, in
which case we always assume that all differentials are of degree zero.
Given
$A, B \in K^b(\BMod{R}{R})$ we denote by $\Hom^\bullet(A,B)$ the
total complex of the double complex with $(i,j)^{th}$-term
$\Hom(A^i,B^j)$ as above. As with bimodules we write
\[
\Hom_K(A,B) = \bigoplus_{i \in \mathbb{Z}} \hom(A, B(i))
\]
for the graded vector space of homotopy classes of morphisms of
complexes of all degrees. We have $\Hom_K(A,B) =
H^0(\Hom^\bullet(A,B))$.

Given $A \in
K^b(\mathcal{A})$ and any $i \in \mathbb{Z}$ we have a distinguished triangle \cite[Exercise 11.2]{KS}
\begin{equation} \label{eq:weighttriangle}
w_{\ge i}A \to A \to w_{< i} A \stackrel{[1]}{\longto}
\end{equation}
where $w_{\ge i}A$ (resp. $w_{< i} A$) denote the ``stupid
truncations'' of $A$: $(w_{\ge i}A)^j = A^j$ if $j \ge i$ and is zero
otherwise, whilst $(w_{< i}A)^j = A^j$ if $j < i$ and is
  zero otherwise.

\subsection{Soergel bimodules} \label{SB}

We start by recalling some aspects of Soergel bimodules, as explained
in \cite{So3} and its relation to Rouquier's categorification of braid
groups, as explained in \cite{Ro1}. Throughout, $(W,\mathcal{S})$
denotes a Coxeter system and $\mathcal{T} = \bigcup_{x \in W}
x\mathcal{S}x^{-1}$ denotes the reflections in $W$. Let $\ell : W \to
\mathbb{N}$ denote the length function and
$\le$ denote the Bruhat order on $W$.

Recall that a representation $V$ of $W$ is said to be reflection
faithful if it is both faithful and has the property that an element $w \in W$ fixes a hyperplane
if and only if it is a reflection. For example, the geometric
representation of any finite Weyl group is reflection faithful,
whereas this is never true for an affine Weyl group.  It is known
\cite{So3} that any Coxeter group has a reflection faithful representation defined
over the real numbers. Throughout, we let $V$ denote a fixed reflection
faithful representation over a field of characteristic $\ne 2$.

We let $R$ denote the regular functions on $V$ which we view as a
graded ring with $\deg V^* = 2$.  Alternatively, we can view $R$ as
the symmetric algebra on $V^*$.

Throughout, we write $B_s := R\otimes_{R^s}R(1)$. The category
$\mathcal{B}$ of Soergel bimodules is the smallest additive monoidal
Karoubian strict subcategory of $\BMod{R}{R}$ which contains $B_s$ for all
$s \in \mathcal{S}$ and is stable under arbitrary shifts. The
reader scared by so many adjectives will 
probably be happier with the following equivalent definition:
$\mathcal{B}$ is the full subcategory of $\BMod{R}{R}$ with objects
those bimodules isomorphic to direct sums of graded shifts of direct
summands of bimodules of the form $B_sB_t \dots
B_u$ for $s, t, \dots, u \in \mathcal{S}$.

\begin{remark}
At various points in our argument it is more convenient to work with a
reflection faithful representation when considering various
filtrations on Soergel bimodules. However, using 
\cite[Th\'eor\`eme 2.2]{li2} one can usually deduce theorems which are also valid for
the geometric representation once one has proven them for a reflection
faithful representation. (The key technical point is that the
reflection faithful representation $V$ defined by Soergel has a
subrepresentation $V_{\textrm{geom}} \subseteq V$ isomorphic to the geometric
representation.) In particular, having established Theorem \ref{A} for
Soergel bimodules built using $V$ it is easy to conclude, using the
results of \cite{li2}, that it also holds if instead we had used
$V_{\textrm{geom}}$.
\end{remark}

\subsection{The 2-braid group}\label{explicit} We will explain in more
detail the construction of Rouquier complexes. Throughout this paper we use
\[
K = K^b(\BMod{R}{R})
\]
to denote the homotopy category of bounded complexes of graded $R$-bimodules.

For a reflection $t\in \mathcal{T}$, consider $\alpha_t\in V^*$ an equation of the hyperplane of $V$ fixed by $t$. The $\alpha_t$ are unique up to non-zero scalar, we fix them arbitrarily. In equations, $\mathrm{ker}(\alpha_t)=V^t.$ 

For $s\in \mathcal{S}$ consider the graded $R$-bimodule morphism $\eta_s:R\rightarrow B_s(1)$ defined by the equation $\eta_s(1)=\frac{1}{2}(1\otimes \alpha_s+\alpha_s\otimes 1)$, and the multiplication morphism $m_s:B_s\rightarrow R(1).$ We define the complexes of graded $R$-bimodules:
$$F_s= \quad  \dots \to 0  \rightarrow B_s
 \buildrel {m_s} \over \longrightarrow
R(1) \rightarrow
0 \to \dots
$$
and
$$F_{s}^{-1}= \quad \dots \to 0  \rightarrow R(-1)
 \buildrel {\eta_s} \over \longrightarrow
B_s \rightarrow
0 \to \dots
$$
where in both complexes $B_s$ sits in complex degree zero (that is
$F_s^0 = (F_{s}^{-1})^0 = B_s$). It is straightforward to verify that in
$K$ we have isomorphisms
\[
F_s^{-1}F_s \cong F_sF_s^{-1} \cong R
\]
which justifies the notation.

Recall that the braid group $B_W$ of the Coxeter system $(W,\mathcal{S})$ is defined by the generators
$\mathbf{{S}}=\{\mathbf{s}\}_{s\in\mathcal{S}}$ and relations 
\[\underbrace{\mathbf{sts}\cdots}_{m_{st}\mathrm{terms}}\cong \underbrace{\mathbf{tst}\cdots}_{m_{st}\mathrm{terms}}.\]
Let $B_W^+$ denote the submonoid of $B_W$ generated by
$\mathbf{S}$.

As explained in the introduction, in \cite{Ro1} Rouquier proves that
for every two decompositions of an element of $B_W$ in a product of
the generators and their inverses there exist a canonical isomorphism in $K$ between the corresponding
product of $F_s$. If $\sigma$ is an element of $B_W$ as in the introduction we denote by $F_{\sigma}$ the corresponding element in
$K$.  
Let $\sigma \in B_W^+$ be the canonical positive lift of $w \in
W$. Then we define $F_w=F_{\sigma}$ and $E_w=F_{w^{-1}}^{-1}.$  

\subsection{Review of category $\mathcal{O}$}\label{2.1}
Here we give a very quick review of the facts that we will need of category $\mathcal{O}.$ For more details see \cite{Hu}. 

Let $\mathfrak{g}\supset \mathfrak{b}\supset \mathfrak{h}$ be
respectively a complex semi-simple Lie algebra, a Borel and Cartan
subalgebra and $W$ the Weyl group. Let $\mathcal{O}$ be the
Bernstein-Gelfand-Gelfand category of finitely generated
$\mathfrak{h}$-diagonalizable and locally $\mathfrak{b}$-finite
$\mathfrak{g}$-modules.  

For all $\lambda\in \mathfrak{h}^*$ we have a standard module, the
\textit{Verma module}
$\Delta(\lambda)=\mathcal{U}(\mathfrak{g})\otimes_{\mathcal{U}(\mathfrak{b})}\mathbb{C}_{\lambda}$,
where $\mathbb{C}_{\lambda}$ denotes the irreducible
$\mathfrak{h}$-module with weight $\lambda$ inflated via the
surjection $\mathfrak{b} \onto \mathfrak{b} / [\mathfrak{b},
\mathfrak{b}] \cong \mathfrak{h}$ to a module over the Borel
subalgebra. The module $\Delta(\lambda)$ has a unique simple quotient
$L(\lambda)$ and we denote $P(\lambda)$ its projective cover. Let $D$
be a duality on $\mathcal{O}$ that fixes simple modules (up to
isomorphism), we put $\nabla(\lambda)=D\Delta(\lambda).$ Recall that $T$ 
is a \textit{tilting object} in the category $\mathcal{O}$
if and only if $T$ and $DT$ have filtrations by $\Delta$-modules.

The dot-action of the Weyl group $W$ on $\mathfrak{h}^*$
is given by $w\cdot   \lambda=w(\lambda+\rho)-\rho,$ where $\rho$ is
the half-sum of the positive roots. For $\lambda\in \mathfrak{h}^*$ we
denote by $\mathcal{O}_{\lambda}$ the full subcategory of
$\mathcal{O}$ with objects the modules killed by some power of the
maximal ideal $\mathrm{Ann}_{\mathcal{Z}}\Delta(\lambda)$ of the
center $\mathcal{Z}$ of $\mathcal{U}(\mathfrak{g})$. This yields
a decomposition $$\mathcal{O}=\bigoplus_{\lambda\in
  \mathfrak{h}^*/(W\cdot)}\mathcal{O}_{\lambda}.$$The
subcategories $\mathcal{O}_{\lambda}$ corresponding to integral weights are
indecomposable, though this is not true for general weights.

Consider $\mu,\lambda\in\mathfrak{h}^*$ such that $\lambda-\mu$ is an
integral weight. Let $E(\mu-\lambda)$ be the finite dimensional
irreducible $\mathfrak{g}$-module with extremal weight
$\mu-\lambda$. Then, we define the translation functor
$T_{\lambda}^{\mu}:\mathcal{O}_{\lambda}\rightarrow \mathcal{O}_{\mu},
M\mapsto \mathrm{pr}_{\mu}(E(\mu-\lambda)\otimes M)$, where
$\mathrm{pr}_{\mu}$ is the projection functor onto the block
$\mathcal{O}_{\mu}.$ The functors $T_{\lambda}^{\mu}$ and
$T_{\mu}^{\lambda}$ are left and right adjoints to each other. 

If we choose $\mu$ a singular weight with $\{1,s\}$ the stabilizer
under the dot-action of the Weyl group, and $s$ a simple reflection,
we can define the wall-crossing functor as the composition
$\theta_s=T_{\mu}^{0}T_{0}^{\mu}.$

Let $W$ be a Weyl group and $R$ as in Section \ref{SB}. Let
$C=R/(R_+^W)$ be the coinvariant algebra. We identify $R/R_+$ with the
ground field $\mathbb{C}.$ In \cite[Endomorphismensatz]{So-1} Soergel
proves that there exists an isomorphism $C\cong
\mathrm{End}_{\mathcal{O}}(P(w_0\cdot 0))$, and in \cite[\S 2.3]{MS2} or \cite{BeBeMi} it is
proved that  the exact
functor $$\mathbb{V}=\mathrm{Hom}_{\mathcal{O}}(P(w_0\cdot
0),-):\mathcal{O}_0\rightarrow C-\mathrm{mod} $$
is fully faithful on tilting objects (this is an alternate version of
\cite[Struktursatz 9]{So-1}). From now on, the objects $P(w\cdot 
0), \Delta(w\cdot 0)$ and $\nabla(w\cdot 0)$ will be denoted $P(w),
\Delta(w)$ and $\nabla(w)$.

\subsection{Rouquier's formula for Weyl groups}\label{1.4}
In this Section we prove Theorem \ref{A} for Weyl groups using formula
\eqref{ext} in the introduction. Of course it is a corollary of Theorem \ref{A} proved in Section \ref{proof}, but we consider this section important to understand where this formula comes from. Using the notation we have introduced we restate Rouquier's formula.
 If $w,v\in W$ we have the following isomorphisms
\begin{equation*}
\mathrm{Hom}_{K}(F_w,E_{v}[i])\cong
\begin{cases}
R& \text{if } w=v \text{ and } i=0,\\
0& \text{otherwise.}
\end{cases}
\end{equation*}

 Consider the bounded complex of functors from $\mathcal{O}_0$ to itself, 
\[
\Phi_s=\quad \dots \to 0\to \theta_s\to \Id\to 0 \to \dots
\]
where $\theta_s$ is in degree zero and the map from $\theta_s$ to the identity functor is the counit of the adjunction. Since the functors involved are all exact, this complex defines an exact functor
$$\Phi_s : K^b(\mathcal{O}_0)\rightarrow K^b(\mathcal{O}_0)$$
by applying the complex of functors to a complex of modules and then taking the total complex of the resulting double complex.

There is another bounded complex of functors, 
$$\Psi_s=\quad \dots \to 0\to \Id\to\theta_s\to 0 \to \dots$$
giving a functor $\Psi_s$ from $K^b(\mathcal{O}_0)$ to itself, where again $\theta_s$ is in degree zero, and this time the map from the identity functor to $\theta_s$ is the unit of the adjunction. 

 If $w=s_1\cdots s_n$  is a reduced expression of an element of  
$W$, we put $\Psi_w=\Psi_{s_1}\cdots\Psi_{s_n}$ and
$\Phi_w=\Phi_{s_1}\cdots\Phi_{s_n}$. (This is well defined because the
functors $\Phi_s$ and $\Psi_s$ satisfy the braid relations.)

\begin{lem}\label{til} If  $w_0$ is the longest element of $W$ and $w\in W$, then the complex $\Phi_w(\Delta(w_0))$  is a tilting resolution of $\Delta(w_0w^{-1})$ and the complex $\Psi_w(\nabla(w_0))$ is a tilting resolution of $\nabla(w_0w^{-1})$.
\end{lem}
\begin{proof}
We will prove only the first statement, the second is proven similarly.
 As $\Delta(w_0)$ is simple, and the duality $D$ fixes simple modules,
 we have that $\Delta(w_0)$ is tilting. As wall-crossing functors
 preserve tiltings (see for example \cite[Prop. 11.1 (e)]{Hu}), the
 terms of the complex $\Phi_w(\Delta(w_0))$ are tilting. So we only
 need to prove that $\Phi_w(\Delta(w_0))$  is a  resolution of
 $\Delta(w_0w^{-1})$, or in other words,
 that \begin{equation}\label{Hi}H^i(\Phi_w(\Delta(w_0)))=\delta_{i,0}\Delta(w_0w^{-1}).\end{equation}

 For every $w$ such that $ws>w$, we have the well known exact sequence: 
$$0\to \Delta(w)\to\theta_s\Delta(w)\to \Delta(ws)\to
0$$ 
This sequence combined with the fact that for every $x \in W$ we have
$\theta_s\Delta(x)\cong \theta_s\Delta(xs),$ proves \eqref{Hi}, by
induction on the length of $w$. \end{proof}

We have the following sequence of isomorphisms
\begin{displaymath}
\begin{array}{lll}
\Ext^i_{\CO}(\Delta(w_0w^{-1}),\nabla(w_0v^{-1})) &\simeq &   \Hom_{K^b(\CO)}(\Phi_w(\Delta(w_0)),\Psi_{v}(\nabla(w_0))[i])\\
&\simeq &    \Hom_{K^b(C)}(F_{w}\otimes \mathbb{C},F_{v^{-1}}^{-1}\otimes \mathbb{C}[i])  \\
&\simeq & \Hom_{K^b(R\otimes R)}(F_w,F_{v^{-1}}^{-1}[i])\otimes_R \mathbb{C} .
\end{array}
\end{displaymath}
The first isomorphism follows from Lemma \ref{til} and the fact that if $M$ and $N$ are tilting objects, then $\mathrm{Ext}^i_{\CO}(M,N)=0$ for all $i>0$ (see for example \cite[Prop. 11.1 (f)]{Hu}). 

For the second isomorphism we apply $\mathbb{V}$ and remark that on
the one hand we have isomorphisms of functors from
$K^b(\mathcal{O}_0)$ to $K^b(C)$:  $\mathbb{V}\Phi_s\simeq
F_s\mathbb{V}$ and $\mathbb{V}\Psi_s\simeq F_s^{-1}\mathbb{V}$ (see \cite[Theorem 10]{So-1}), and
on the other $\mathbb{V}(\nabla(w_0))\simeq
\mathbb{V}(\Delta(w_0))\simeq \mathbb{C}$ (which follows because
$\nabla(w_0) \cong \Delta(w_0)$ is simple). This together with the fact
that  
 $\mathbb{V}$ is fully faithful on tilting objects  proves the second isomorphism. The third isomorphism is a consequence of \cite[Prop. 8]{So0}.

The graded Nakayama Lemma and \eqref{ext} of the introduction allows
us to conclude the proof of the second case of Theorem \ref{Princ},
and the first case ($w=v$ and $i=0$) follows easily by the adjunctions
arguments as $F_s$ is right and left adjoint of $F_{s^{-1}}$ (see
\cite{Ri} or \cite[Lemma 5.1 and Corollary 5.5]{MS2}).

\section{Arbitrary Coxeter groups}\label{proof}\label{sosho}

In this Section we prove Rouquier's formula for arbitrary Coxeter groups. This involves a more detailed look at Soergel bimodules, and in particular their behaviour under various natural subquotient functors.

\subsection{Filtration by support}\label{ca}

For $x\in W$ consider the reversed graph$$\Gr(x)=\{(xv,v)\,\vert\, v\in V\}\subseteq V\times V$$
and for any subset $A$ of $W$ consider the union of the corresponding graphs
$$\Gr(A)=\bigcup_{x\in A}\Gr(x)\subseteq V\times V.$$
Given a finite set $A \subset W$ we can view 
$\Gr(A)$ as a subvariety of $V \times V$. If we identify
$R\otimes R$ with the regular functions on $V\times V$ then $R_A$,
the regular functions on $\Gr(A)$, is naturally a $\mathbb{Z}$-graded
$R$-bimodule. If $A=\{x_1,\ldots, x_n\}$ we will also write
$R_A=R_{x_1,\ldots, x_n}.$ For example, one may check that given $x
\in W$ the bimodule $R_x$ has the following simple description: $R_x
\cong R$ as a left module, and the right action is twisted by $x$: $m
\cdot r = mx(r)$ for $m \in R_x$ and $r \in R$.

For any $R$-bimodule $M \in \BMod{R}{R}$ we can view $M$ as an $R
\otimes R$-module (because $R$ is commutative) and hence as a
quasi-coherent sheaf on $V \times V$. Given any subset $A \subseteq W$
(not necessarily finite) we define
\[
\Gamma_A M := \{ m \in M \; | \; \supp m \subseteq \Gr(A) \}
\]
to be the subbimodule consisting of elements whose support is contained in $\Gr(A)$. Note that $\Gamma_A$ is an endofunctor of $\BMod{R}{R}$.

In the following we will abuse notation and write $\le x$ for the set
$\{ y \in W \; | \; y \le x \}$ and similarly for $<x$, $\ge x$ and $>
x$. With this notation, we obtain functors $\Gamma_{\le x}$, $\Gamma_{< x}$,
$\Gamma_{\ge x}$ and $\Gamma_{> x}$. For example $\Gamma_{\le x} =
\Gamma_{\{ y \in W \; | \; y \le x \}}$.

In the sequel an important role will be played by the ``subquotient functors'':
\begin{gather*}
  \Gamma_{\ge x / > x} : M \mapsto \Gamma_{\ge x}M / \Gamma_{> x}M \\
 \Gamma_{\le x / < x} : M \mapsto \Gamma_{\le x}M / \Gamma_{< x}M \\
\end{gather*}Let $M \in \BMod{R}{R}$ and assume that $M$ is finitely generated as
an $R$-bimodule, and that the support of $M$ is contained in $\Gr(A)$
for some finite set $A \subseteq W$. We say that $M$ has a \emph{$\Delta$-filtration} and write $M \in \mathcal{F}_{\Delta}$ if for all $x \in W$,
$\Gamma_{\ge x/>x} M$ is isomorphic to a direct sum of shifts of $R_x$. Similarly, we say that $M$ has a \emph{$\nabla$-filtration} and write $M \in \mathcal{F}_{\nabla}$ if for all $x \in W$,
$\Gamma_{\le x/<x} M$ is isomorphic to a direct sum of shifts of
$R_x$. We call the full subcategories $\mathcal{F}_{\Delta}$ and
$\mathcal{F}_\nabla$ of $\BMod{R}{R}$ the categories of \emph{bimodules with
  $\Delta$-flag} and \emph{bimodules with $\nabla$-flag} respectively.

A first important result about the categories $\mathcal{F}_{\Delta}$
and  $\mathcal{F}_{\nabla}$ is Soergel's ``hin-und-her Lemma''
\cite[Lemma 6.3]{So3}. It states that given any enumeration $w_0, w_1,
\dots $ of the elements of $W$ compatible with the Bruhat
order (i.e. $w_i \le w_j \Rightarrow i\le j$) then $M
\in \BMod{R}{R}$ belongs to $\mathcal{F}_\nabla$ if and only if for
all $i$ the subquotient
\[
\Gamma_{\{ w_0, \dots, w_i\}} M /
\Gamma_{\{ w_0, \dots, w_{i-1} \}} M
\]
is isomorphic to a direct sum of
shifts of $R_{w_i}$, in which case the natural map
\begin{equation} \label{lem:hinundher1}
\Gamma_{\le w_i / < w_i} M \to \Gamma_{\{ w_0, \dots, w_i\}} M /
\Gamma_{\{ w_0, \dots, w_{i-1} \}} M
\end{equation}
is an isomorphism. Similarly, $M \in \mathcal{F}_{\Delta}$ if and
only if for all $i$,
\[
\Gamma_{\{ w_i, w_{i+1},\dots\} \cap A} M /
\Gamma_{\{ w_{i+1}, \dots \} \cap A} M
\]
are isomorphic to direct sums
of shifts of $R_{w_i}$, in which case the natural map 
\begin{equation} \label{lem:hinundher2}
\Gamma_{\ge w_i / > w_i} M \to \Gamma_{\{ w_i, w_{i+1},\dots\} \cap A} M /
\Gamma_{\{ w_{i+1}, \dots \} \cap A} M
\end{equation}
is an isomorphism (here $A \subseteq W$ denotes a finite set such that
$\supp M \subseteq \Gr(A)$).

The following lemma (one of the first consequences of the hin-und-her
Lemma) implies that if $M \in \mathcal{F}_\nabla$ then so are $\Gamma_{\le
  x} M$ and $M/ \Gamma_{\le x} M$.

\begin{lem} Let $M \in
\mathcal{F_\nabla}$. \label{nablaproj}
\begin{enumerate}
\item $\Gamma_{\le x / < x} (\Gamma_{\le y} M)$ is
zero unless $x \le y$ in which case the natural map
\begin{equation}
\Gamma_{\le x / < x} (\Gamma_{\le y} M) \to \Gamma_{\le x / < x} M
\end{equation}
is an isomorphism.
\item $\Gamma_{\le x/<x}(M/\Gamma_{\le y} M)$ is
zero unless $x \not \le y$ in which case the natural map
\begin{equation}
\Gamma_{\le x / < x} M \to \Gamma_{\le x/<x}(M/\Gamma_{\le y} M)
\end{equation}
is an isomorphism.
\end{enumerate}
\end{lem}

\begin{proof} Statement (1) is straightforward. Let us prove (2). 
Choose an enumeration $w_0, w_1, \dots$ of the
  elements of $W$ compatible with the Bruhat order and such that $\{
  \le y \} = \{ w_0, w_1, \dots, w_k \}$ for some $k$. Let us
  abbreviate $\Gamma_{\le i} = \Gamma_{\{ w_0, \dots, w_i\}}$,
  $\Gamma_{< i} = \Gamma_{\{ w_0, \dots, w_{i-1}\}}$ and 
\[ \Gamma_{\le  i / < i} (M) = \Gamma_{\le i} (M) /\Gamma_{< i}(M). \]
By the hin-und-her lemma we see that (2) is equivalent to the
statement:
\begin{enumerate}
\item[(2')] $\Gamma_{\le i/<i}(M/\Gamma_{\le k} M)$ is
zero unless $i \not \le k$ in which case the natural map
\begin{equation}
\Gamma_{\le i / < i} M \to \Gamma_{\le i/<i}(M/\Gamma_{\le k} M)
\end{equation}
is an isomorphism.
\end{enumerate}
To prove (2') consider the
filtration
\[ \dots \subseteq F_i \subseteq F_{i+1} \subseteq \dots \]
on $M / \Gamma_{\le k} M$ obtained by taking the image of the
filtration
\[
\dots  \subseteq \Gamma_{\le i}M \subseteq \Gamma_{\le i+1} M
\subseteq \dots \]
on $M$. By the third isomorphism theorem
\[
F_i/F_{i-1} \cong \begin{cases} 0 & \text{if $i \le k$} \\ \Gamma_{\le
    i / < i} M & \text{otherwise.}  \end{cases}
\]
By the lemma below we have $F_i = \Gamma_{\le i} (M/\Gamma_{\le
  k}M)$. Statement (2') now follows.
\end{proof} 

\begin{lem} \label{lem:filt} 
Let $N \in \BMod{R}{R}$ and fix an enumeration $w_0, w_1, \dots$ of
the elements of $W$ compatible with the Bruhat order. Suppose that we
have a filtration $0 = N_{-1} \subseteq N_0 \subseteq N_1 \subseteq
\dots \subseteq N_n = N$ of $N$ such that $N_i / N_{i-1}$ is
isomorphic to a direct sum of shifts of $R_{w_i}$ for $0 \le i \le
n$. Then $N_i = \Gamma_{\{w_0, w_1, \dots, w_i\}} N$.
\end{lem}

\begin{proof}
  Because $N_i$ is an extension of shifts of $R_{w_l}$ with $l \le i$ we have $N_i \subseteq \Gamma_{\{ w_0, \dots, w_i \} } N$. It remains the prove the reverse inclusion. So choose $m \in \Gamma_{\{ w_0, \dots, w_i \} }N$ and let $k$ be minimal such that $m \in N_k$ but $m \notin N_{k-1}$. Then the image of $m$ in $N_k /N_{k-1}$ is non-zero. Using that $N_k /N_{k-1}$ is isomorphic to a direct sum of shifts of $R_{w_k}$, and that any non-zero element of $R_{w_k}$ has support equal to $\Gr(w_k)$ we see that $\Gr(w_k)$ is contained in the support of $m$. Hence $w_k \in \{ w_0, \dots, w_i \}$, so $k \le i$, so $m \in N_i$. Hence $\Gamma_{\{ w_0, \dots, w_i \} } N \subseteq N_i$.
\end{proof}

The following lemma is the $\Delta$-version of
Lemma \ref{nablaproj}. It implies that if $M \in \mathcal{F}_\Delta$ then so are $\Gamma_{\ge x} M$ and $M
/\Gamma_{\ge x} M$ for any $x \in W$. The proof (which we omit) is
similar to that of Lemma \ref{nablaproj}.

\begin{lem} Let $M \in
\mathcal{F}_\Delta$. \label{deltaproj}
\begin{enumerate}
\item $\Gamma_{\ge x / > x} (\Gamma_{\ge y} M)$ is
zero unless $x \ge y$ in which case the natural map
\begin{equation}
\Gamma_{\ge x / > x} (\Gamma_{\ge y} M) \to \Gamma_{\ge x / > x} M
\end{equation}
is an isomorphism.
\item $\Gamma_{\ge x/>x}(M/\Gamma_{\ge y} M)$ is
zero unless $x \not \ge y$ in which case the natural map
\begin{equation}
\Gamma_{\ge x / > x} M \to \Gamma_{\ge x/>x}(M/\Gamma_{\ge y} M)
\end{equation}
is an isomorphism.
\end{enumerate}
\end{lem}

In \cite[Section 5]{So3} Soergel proves that Soergel bimodules
belong to both $\mathcal{F}_\Delta$ and $\mathcal{F}_\nabla$. Another
important result is Soergel's Hom formula \cite[Theorem 5.15]{So3}: for
any $M \in \mathcal{F}_\Delta$ and $N \in \mathcal{B}$, or $M \in
\mathcal{B}$ and $N \in \mathcal{F}_\nabla$, one has an isomorphism of
graded right $R$-modules 
\begin{equation} \label{eq:SH}
\Hom(M,N) \cong \bigoplus_{x \in W} \Hom( \Gamma_{\ge x/>x} M, \Gamma_{\le x/<x} N)(-2\ell(x)).
\end{equation}

Finally, it is natural to ask how the support of a bimodule changes under the functor $M \mapsto MB_s$. It is straightforward (see \cite[Lemma 4.14]{Wi}) to show that if $M \in \BMod{R}{R}$ has support contained in $\Gr(A)$ for some finite subset $A \subseteq W$ then we have\[
\supp (MB_s) \subseteq \Gr(A \cup As) \quad \text{and} \quad \supp (B_sM) \subseteq \Gr(A \cup sA),
\]
It follows that if $s, t, \dots, u \in \mathcal{S}$ then
\begin{equation} \label{BSsupport}
\supp (B_sB_t \dots B_u) \subseteq \Gr(\{ \id, s\} \{ \id, t\} \dots \{ \id, u \} ).
\end{equation}

\begin{remark}
This terminology $\Delta$-filtered and $\nabla$-filtered is intended to remind the reader of category
$\mathcal{O}$: $\mathcal{F}_\Delta$ (resp. $\mathcal{F}_\nabla$)
can be thought of as those objects with standard (resp. costandard)
filtrations. The above results show that Soergel bimodules can be
thought of as akin to tilting modules. As is well-known (and follows
from the vanishing formula \eqref{ext} of the introduction), the functor of
homomorphisms from (resp. to) a tilting module is exact on complexes of costandard (resp. standard)
filtered objects. It is this analogy which motivates the introduction
of $\Delta$- and $\nabla$-exact complexes in the next section.
\end{remark}

\subsection{$\Delta$- and $\nabla$-exact complexes}\label{delta}

The following definition is fundamental to our proof of Rouquier's formula:

\begin{defi}
  Let $A$ be a bounded complex of graded $R$-bimodules.
  \begin{enumerate}
  \item $A$ is \emph{$\Delta$-exact} if $A^i \in \mathcal{F}_{\Delta}$ for all $i$ and, for all $x \in W$ the complex $\Gamma_{\ge x/>x} A$ is homotopic to zero.
  \item $A$ is \emph{$\nabla$-exact} if $A^i \in \mathcal{F}_{\nabla}$ for all $i$ and, for all $x \in W$ the complex $\Gamma_{\le x/<x} A$ is homotopic to zero.
  \end{enumerate}
\end{defi}

\begin{remark} \hspace{2cm}
  \begin{enumerate}
\item Because an extension of acyclic complexes is acyclic it follows that any $\Delta$- or $\nabla$-exact complex is acyclic.
\item The canonical example of a $\Delta$-exact complex is the 3-term complex
\begin{equation} \label{eq:deltaexact}
0 \to \Gamma_{\ge y}M \to M \to M/\Gamma_{\ge y} M \to 0
\end{equation}
for some $M \in \mathcal{F}_{\Delta}$. That this sequence is
$\Delta$-exact follows from Lemma \ref{deltaproj}.

\item Similarly, if $M \in \mathcal{F}_\nabla$ then the sequence
\begin{equation} \label{eq:nablaexact}
0 \to \Gamma_{\le y}M \to M \to M/\Gamma_{\le y} M \to 0
\end{equation}
is $\nabla$-exact. This follows from Lemma \ref{nablaproj}.

\item If $A$ is $\Delta$-exact then so are $\Gamma_{\ge x}A$ and $A /
  \Gamma_{\ge x} A$ for any $x \in W$. Similarly, if $A$ is
  $\nabla$-exact then so are $\Gamma_{\le x}A$ and $A / \Gamma_{\le x}
  A$ for any $x \in W$. Again these statements follow easily from Lemmas
  \ref{nablaproj} and \ref{deltaproj}.
  \item If $A$ is a bounded complex of modules belonging to
    $\mathcal{F}_{\Delta}$ then $\Gamma_{\ge x/>x} A^i$ is isomorphic
    to a direct sum of shifts of $R_x$ for all $x \in W$ and $i \in
    \mathbb{Z}$. One may show that a bounded complex of modules all of
    whose terms are isomorphic to direct sums of shifts of $R_x$'s is
    homotopic to zero if and only if it is acyclic. (Such a complex of
    $R$-bimodules is homotopic to zero if and only if it is homotopic
    to zero as a complex of left $R$-modules. However each $R_x$ is
    projective as a left $R$-module, and so any acyclic complex
    splits.) Hence we could have replaced ``homotopic to zero'' by
    ``exact'' in the above definition. This also explains the origin of the terminology.
  \end{enumerate}
\end{remark}

\begin{ex}
Consider the complexes
\begin{gather*}
\widetilde{F_s} = \quad \dots \to R_s(-1) \stackrel{\eta_s'}{\to} B_s
\stackrel{m_s}{\to} R(1) \to \dots \\
\widetilde{E_s} = \quad \dots \to R(-1) \stackrel{\eta_s}{\to} B_s
\stackrel{m_s'}{\to} R_s(1)  \to \dots
\end{gather*}
where $m_s$ and $\eta_s$ are as in Section \ref{explicit} and $\eta_s'$ and
$m_s'$ are determined by $\eta_s'(1) = \frac{1}{2}(\alpha_s\otimes 1 - 1 \otimes
\alpha_s)$ and $m_s'(f \otimes g) = fs(g)$ respectively.
We claim that $\widetilde{F_s}$ is $\Delta$-exact, but that
$\widetilde{E_s}$ is not.
Indeed, one has isomorphisms
\begin{gather*}
\Gamma_{\ge s} \widetilde{F_s}\cong \quad \dots \to R_s(-1)
\stackrel{\id}{\longto} R_s(-1) \to 0 \to \dots \\
\Gamma_{\ge s} \widetilde{E_s}\cong \quad \dots \to 0 \to R_s(-1)
\stackrel{\alpha_s \cdot}{\longto} R_s(1)  \to \dots
\end{gather*}
and
\begin{gather*}
\Gamma_{\ge \id / \ge s} \widetilde{F_s}\cong \quad \dots \to 0 \to R(1)
\stackrel{\id}{\longto} R(1) \to \dots \\
\Gamma_{\ge \id / \ge s} \widetilde{E_s}\cong \quad \dots \to R(-1)
\stackrel{\alpha_s \cdot}{\longto} R(1) \to 0 \to \dots
\end{gather*}
 A dual calculation shows that
$\widetilde{E_s}$ is $\nabla$-exact and that $\widetilde{F_s}$ is
not.

In fact, $\widetilde{F_s}$ and $\widetilde{E_s}$ are examples of
augmented Rouquier complexes which will be introduced in Section
\ref{2.2}. The above calculations give a concrete example of Proposition \ref{h}.
\end{ex}

The following lemma is the central technical tool of this paper. It shows that $\Delta$- (resp. $\nabla$-) exact complexes are acyclic for the functor of homomorphisms to (resp. from) a Soergel bimodule:

\begin{prop} \label{prop:exact}
Let $A$ denote a bounded complex of $R$-bimodules and let $B \in \mathcal{B}$ be a Soergel bimodule:
  \begin{enumerate}
  \item if $A$ is $\Delta$-exact then $\Hom_K(A, B)=0$.
\item if $A$ is $\nabla$-exact then $\Hom_K(B, A)=0$.
  \end{enumerate}
\end{prop}

\begin{proof}
  We only prove 1). Statement 2) is proved by a very similar argument.

Suppose first that $A$ is a $\Delta$-exact complex consisting of only three non-zero terms:
\[
0 \to M^1 \to M^2 \to M^3 \to 0.
\]
Let us apply $\Hom(-, B)$. Because $\Hom(-,B)$ is left exact we have an exact sequence
\begin{equation} \label{eq:lemseq}
0 \to \Hom(M^3, B) \to \Hom(M^2, B) \stackrel{r}{\to} \Hom(M^1, B).
\end{equation}
We claim that $r$ is in fact surjective. Let us consider \eqref{eq:lemseq} in each degree separately. So fix a degree $i \in \mathbb{Z}$. Certainly we have an exact sequence of finite dimensional vector spaces
\[
0 \to \Hom(M^3, B)_i \to \Hom(M^2, B)_i \stackrel{r_i}{\to} \Hom(M^1, B)_i.
\]
On the other hand, by the fact that our sequence is $\Delta$-exact we
have, for any $x \in W$,
\[
 \Gamma_{\ge x/>x} M^2 \cong  \Gamma_{\ge x/>x} M^1 \oplus  \Gamma_{\ge x/>x} M^3
\]
Soergel's Hom formula \eqref{eq:SH} gives:
\[
\dim \Hom(M^3, B)_i + \dim \Hom(M^1, B)_i = \dim \Hom(M^2, B)_i \;.
\]
We conclude that $r_i$ is surjective. Hence $r$ is surjective as claimed and indeed we have an exact sequence
\[
0 \to \Hom(M^3, B) \to \Hom(M^2, B) \to \Hom(M^1, B) \to 0.
\]

We now turn to the general case and argue by induction on the size of the set
\[
\supp A := \{ x \in W \; | \; \text{there exists an $i$ such that } \Gamma_{\ge x/>x} A^i \ne 0 \}.
\]
If $|\supp A| = 1$ then $A$ is homotopic to zero (as follows directly from the definition of $\Delta$-exactness) and the lemma holds in this case.

For the general case, fix $x \in \supp A$ which is maximal in the Bruhat order and consider the sequence of $\Delta$-exact complexes
\[
0 \to \Gamma_{\ge x} A \to A \to A / \Gamma_{\ge x} A \to 0.
\]
By the remarks following the definition of $\Delta$-exact each row and
column of this sequence is $\Delta$-exact.

By the 3-term case considered above, applying $\Hom(-,B)$ yields an exact sequence of complexes
\[
0 \to \Hom^\bullet(A / \Gamma_{\ge x} A,B) \to \Hom^\bullet(A,B) \to \Hom^\bullet(\Gamma_{\ge x} A,B) \to 0
\]
Now $\supp(\Gamma_{\ge x} A) = \{ x \}$ and $|\supp(A / \Gamma_{\ge x}
A)| < |\supp A|$. Hence we can apply induction to conclude that the
complexes $\Hom^\bullet(A / \Gamma_{\ge x} A,B)$ and $\Hom^\bullet(\Gamma_{\ge x}
A,B)$ are acyclic. Hence $\Hom^\bullet(A,B)$ is acyclic too, by the long exact
sequence of cohomology. Hence $\Hom_K(A,B) = H^0(\Hom^\bullet(A,B)) =
0$ as claimed.
\end{proof}

\begin{cor}\label{T}
If $F\in K^b(\mathcal{F}_{\Delta})$ is $\Delta$-exact and $G\in K^b(\mathcal{B})$ or $F\in K^b(\mathcal{B})$ and $G\in K^b(\mathcal{F}_{\nabla})$ is $\nabla$-exact, then we have $\mathrm{Hom}_K(F,G)=0$
\end{cor}

\begin{proof} We handle the case of $F \in
  K^b(\mathcal{F}_\Delta)$ and $G \in K^b(\mathcal{B})$. The dual case is analogous.

  We have seen in Proposition \ref{prop:exact} above that if $F$ is
  $\Delta$-exact then $\Hom_K(F,B) = 0$ for any Soergel bimodule $B
  \in \mathcal{B}$. We prove the proposition by induction on 
$\ell(G) := |\{ i \in \mathbb{Z} \; | \; G^i \ne 0 \}|$. The
  case $\ell(G) = 0$ is trivial and the case $\ell(G) = 1$ follows by
  the above proposition.

So fix $G \in K^b(\mathcal{B})$ and assume that we have proven the
lemma for all complexes of Soergel bimodules $G'$ with $\ell(G') <
\ell(G)$. Choose $i$ maximal with $G^i \ne 0$. We have a distinguished
triangle (see \eqref{eq:weighttriangle})
\begin{equation} \label{eq:dt}
w_{\ge i}G \to G \to w_{<i}G \stackrel{[1]}{\to}
\end{equation}
with $\ell(w_{\ge i}G) = 1$ and $\ell(w_{<i}G) =
\ell(G) -1$. As $\hom(F,-)$ is cohomological, so is $\Hom(F,-)$. 
If we apply $\Hom(F,-)$ to \eqref{eq:dt} then we have a long
exact sequence
\[
\dots \to \Hom(F, w_{\ge i}G) \to \Hom(F, G) \to \Hom(F, w_{< i}G) \to \dots
\]
and induction allows us to conclude that $\Hom(F, G) = 0$ as claimed.
\end{proof}

\subsection{Exactness properties of Rouquier complexes} \label{1.2}

In this section we prove the complexes $F_x$ (resp. $E_x$) for $x \in
W$ are almost $\Delta-$ (resp. $\nabla-$) exact. More precisely
the goal is to prove:

\begin{prop} \label{prop:almostsplit}
  For $x, y \in W$ we have isomorphisms in the homotopy category
\[
\Gamma_{\le y/ < y} E_x \cong \begin{cases} R_x(\ell(x)) & \text{if $y =
    x$,}\\
0 & \text{otherwise.} \end{cases}
\]
Similarly,
\[
\Gamma_{\ge y/>y} F_x \cong \begin{cases} R_x(-\ell(x)) & \text{if $y =
    x$,}\\
0 & \text{otherwise.} \end{cases}
\]
\end{prop}

\begin{proof}We prove the first isomorphism, the second follows by a
  dual argument.
We now prove this isomorphism by induction on $\ell(x)$. It can be
verified by hand for $\ell(x) = 0, 1$. So fix $x$ and choose $s \in S$
with $xs < x$ so that $E_x = E_{xs}E_s$.
 By induction we may assume that the first 
isomorphism in the proposition holds for $E_{xs}$. Our aim is to show
that it also holds for $E_x$.

Let us choose an enumeration $w_0, w_1, w_2, \dots$ of the elements of
$W$ compatible with the Bruhat order and such that $w_{m+1} = w_m{s}$
for all even $m$. Such an enumeration can be constructed by first
choosing an enumeration of $W/\langle s \rangle$ compatible with the
Bruhat order and then refining it to an enumeration of $W$.
It follows from 
\cite[Prop. 6.5]{Wi} that for any Soergel bimodule $B$ the
natural map gives an isomorphism:
\[
(\Gamma_{\le m+1/ < m} B)B_s \stackrel{\sim}{\longrightarrow}
(\Gamma_{\le m+1/<m} BB_s) \quad \text{for $m$ even.}
\]
(We use the same notation as in the proof of
Lemma \ref{nablaproj}.) Hence for any complex $F \in K^b(\mathcal{B})$
of Soergel bimodules we have an isomorphism:
\begin{equation} \label{eq:supportsub}
(\Gamma_{\le m+1/ < m} F)E_s \stackrel{\sim}{\longrightarrow}
(\Gamma_{\le m+1/<m} FE_s) \quad \text{for $m$ even.}
\end{equation}

We first deal with the case of $y \in \{ x, xs\}$. Fix $n$ such
that $w_{n+1} = x$. Then $n$ is necessarily even and $w_n = xs$.
If we apply induction and \eqref{eq:supportsub} we obtain
\[
\Gamma_{\le n+1/<n} (E_{x}) \cong \Gamma_{\le n+1/<n} (E_{xs})E_s
\cong \quad \dots \to 
R_{xs}(\ell(x)-2) \to R_{x,xs}(\ell(x)) \to \dots
\]
where $R_{x,xs}(\ell(x))$ occurs in complex degree 0 and all
terms which are not displayed are zero. Indeed, by induction we have
$\Gamma_{\le n+1/<n} (E_{xs}) \cong R_{xs}(\ell(xs))$ and $R_{xs}B_s \cong
R_{xs,x}(1)$.

The result in this case then follows by applying $\Gamma_{\le n+1/<n+1}$ and
$\Gamma_{\le n/<n}$ and using the isomorphisms of functors on
$K^b(\mathcal{F}_\nabla)$ (valid for any $i$):
\begin{align}
\Gamma_{\le i / <
  i}(-) & \cong \Gamma_{\le i}(\Gamma_{\le i+1/<i} (- )), \label{func1}\\
 \Gamma_{\le i+1/<i+1}(-) & \cong \Gamma_{\le i+1/<i+1}(\Gamma_{\le i+1/<i}(-)). \label{func2}
\end{align}
(Both of these isomorphisms follow from Lemma \ref{lem:hinundher2}.)

We now deal with the case $y \notin \{ x, xs \}$. We may assume $ys <
y$ and fix $p$ (again even) such that $y_p = ys$ and
$y_{p+1} = y$. As above we have
\[
\Gamma_{\le p+1/<p} (E_{x}) = \Gamma_{\le p+1/<p} (E_{xs}E_s)\cong
\Gamma_{\le p+1/<p} (E_{xs})E_s.
\]
Consider the exact sequence of complexes of bimodules
\[
0 \to \Gamma_{\le p/<p}E_{xs}\to \Gamma_{\le p+1/<p}E_{xs}\to \Gamma_{\le
  p+1/<p+1}E_{xs}\to 0.
\]
Each term on the left (resp. right) is isomorphic to a direct sum of
shifts of $R_{ys}$ (resp. $R_{y}$). By \cite[Lemma 5.8]{So3} any such
extension splits when we restrict to $\BMod{R}{R^s}$. Hence we have an
exact sequence of complexes
\[
0 \to (\Gamma_{\le p/<p}E_{xs})B_s \to (\Gamma_{\le p+1/<p}E_{xs})B_s \to (\Gamma_{\le
  p+1/<p+1}E_{xs})B_s \to 0
\]
where each row is split exact. Hence we have an exact triangle in
$K^b(\BMod{R}{R})$ 
\[
(\Gamma_{\le p/<p}E_{xs})B_s \to \Gamma_{\le p+1/<p}(E_{xs} B_s) \to (\Gamma_{\le
  p+1/<p+1}E_{xs})B_s \triright.
\]
By induction the left and right hand terms are 
homotopic to zero. We
conclude that $\Gamma_{\le p+1/<p}(E_{xs}B_s)$ is also homotopic to zero and hence
\[
\Gamma_{\le p+1/<p+1}(E_{xs}B_s) \cong \Gamma_{\le p/<p}(E_{xs}B_s) \cong 0
\]
again using \eqref{func1} and \eqref{func2}.

Finally, we have a distinguished triangle
\[
E_{xs}(-1) \to E_{xs}B_s \to E_{xs}E_s \triright
\]
and applying $\Gamma_{\le p+1/<p+1}(-)$ and $\Gamma_{\le p/<p}(-)$ the
two left hand terms are zero. Hence
\[
\Gamma_{\le p+1/<p+1}(E_{x})  \cong \Gamma_{\le p/<p}( E_{x}) \cong 0
\]
which is what we wanted to show.
\end{proof}

\subsection{Augmented Rouquier complexes}\label{2.2}

Given any braid $\sigma = s_1^{m_1}s_2^{m_2} \dots s_n^{m_n} \in B_W$
let $\epsilon(\sigma) = \sum_{i = 1}^n m_i$.

The following is standard:

\begin{lem} Let $\sigma \in B_W$ and let $w$ denote its image in $W$.
We have
\[
H^i(F_\sigma) = \begin{cases} R_w(-\epsilon(\sigma)) & \text{if $i = 0$,} \\
0 & \text{otherwise.} \end{cases}
\]
\end{lem}

\begin{proof}
It is well-known that Soergel bimodules are free when regarded as left or right $R$-modules. (Any Bott-Samelson bimodule $B_sB_t \dots B_u$ is free as a left or right $R$-bimodule and hence so is any direct summand.) It follows that the functor of tensoring on the right or left by a Soergel bimodule is exact.

  We prove the lemma
by induction on the length $\ell(\sigma)$ of a minimal word for
$\sigma$, with the cases $\ell(\sigma) = 0, 1$ following by direct
calculation from the definition of $F_\sigma$. First assume $\sigma = \sigma' s$ for some $\sigma' \in
B_W$ with $\ell(\sigma') = \ell(\sigma) - 1$. The cohomology of
$F_\sigma$ is the cohomology of the double complex $F_{\sigma'} B_s
\to F_{\sigma'} R(1)$. However, by induction and the above remarks, we
get that the cohomology of $F_\sigma$ is the cohomology of the complex
$R_{w} B_s(-\epsilon(\sigma')) \to R_{w} R
(-\epsilon(\sigma')+1)$. Now direct calculation shows that this
complex is quasi-isomorphic to a complex with only non-zero term
$R_{ws}[-\epsilon(\sigma)-1]$ in degree zero, which establishes the
induction step in this case.

The case when $\sigma = \sigma's^{-1}$  for some $s \in \mathcal{S}$ is handled in a similar manner.
\end{proof}

Now fix $w \in W$. The above lemma shows that the cohomology of $F_w$ (resp. $E_w$) is concentrated in degree zero, where it is isomorphic to $R_w(-\ell(w))$ (resp. $R_w(\ell(w))$). From the definitions it is clear that the non-zero terms of the complex $F_w$ (resp. $E_w$) are in degrees $\ge 0$ (resp. $\le 0$). Hence we have truncation morphisms
\[
f_w:R_w(-\ell(w)) = H^0(F_w) = \tau_{\leq 0}F_w\rightarrow F_w
\]
and
\[
e_w: E_w \to \tau_{\ge 0} E_w = H^0(E_w) = R_w(\ell(w)).
\]
We set
\begin{gather}
  \widetilde{F_w} := \cone(f_w) \\
\widetilde{E_w} := \cone(e_w).
\end{gather}
We call these complexes \textit{augmented Rouquier complexes}.

As $F_w$ and $E_w$ only have cohomology in degree zero, the augmented
Rouquier complexes $\widetilde{F_w}$ and $\widetilde{E_w}$ are
exact. In fact, they are exact in a much stronger sense. By
Proposition \ref{prop:almostsplit} we have:

\begin{cor}\label{h}
The complex 
$\widetilde{F_{w}}$ is $\Delta$-exact
and the complex $\widetilde{E_{w}}$ is $\nabla$-exact.
\end{cor}

Combining this result with Corollary \ref{T} yields

\begin{cor}\label{three} For $w,v\in W$ and $H\in K^b(\mathcal{B})$
  and $m \in \mathbb{Z}$ we have
$$\mathrm{Hom}_K(H, \widetilde{E_{v}}[m])=0=\mathrm{Hom}_K(\widetilde{F_{w}}, H[m]).$$
\end{cor}

\subsection{Proof of the main theorem}\label{2.3}
In this final section we will see how $\Delta$-exactness and
$\nabla$-exactness enters the story of the 2-braid group.

\begin{proof}[Proof of Theorem \ref{Princ}]
We have the distinguished triangles in the triangulated category
$K=K^b(\BMod{R}{R})$:
\begin{equation}\label{tri1}
R_w(-\ell(w))\rightarrow F_{w}\rightarrow \widetilde{F_{w}} \stackrel{[1]}{\longto}
\end{equation}
and for any integer $m$:
\begin{equation}\label{tri2}
E_{v}[m]\rightarrow R_{v}(\ell(v))[m]\rightarrow \widetilde{E_{v}}[m]\stackrel{[1]}{\longto}
\end{equation}
By applying the cohomological functor $\Hom_K(-,E_{v}[m])$ to
the triangle \eqref{tri1} we obtain, using Corollary \ref{three}:
\begin{equation}\label{a} \Hom_K(F_{w},E_{v}[m])\cong
  \Hom_K(R_w(-\ell(w)),E_{v}[m]).\end{equation}  
Simlarly, applying $\Hom_K(F_w, -)$ to \eqref{tri2} yields an
isomorphism
\begin{equation}\label{b} \Hom_K(F_{w},E_{v}[m])\cong
  \Hom_K(F_w,R_{v}(\ell(v))[m]).\end{equation} 

If $w = v$ then from \eqref{BSsupport} we have $\supp E_v^i \subset \Gr(<v)$ for $i <
0$ and $\supp E_v^0 = \Gr(\le v)$. Hence 
$\Hom(R_v(-\ell(w)), E_v^i) = 0$ for $i \ne 0$ and (as graded $R$-modules)
\[
\Hom(R_v(-\ell(w)),E_v^0) = \Hom(R_v(-\ell(v)), R_v(\ell(v))(-2\ell(v)) = R
\]
by Soergel's Hom formula \eqref{eq:SH}. Hence the complex
$\Hom^\bullet(R_v(-\ell(v)), E_v)$ is concentrated in degree zero, where it
is isomorphic to $R$. Hence, by \eqref{a},
\[
\Hom(F_v, E_v[m])  = \begin{cases} R & \text{if $m = 0$,} \\
0 & \text{otherwise.} \end{cases}
\]

Now assume that $w \ne v$. We will prove that
\[
\Hom_K(F_w, E_v[m]) = 0
\]
by considering two cases.

For the first case, assume that $w \not \le v$. Using \eqref{a} we
have to show that 
\[
\Hom_K(R_w(-\ell(w)), E_v[m]) = 0.
\]
For all $i \in \mathbb{Z}$, $E_v^i$ is a Soergel
bimodule and $\supp E_v^i \subseteq \Gr(\le v)$ (see \eqref{BSsupport}). It follows that
$E_v^i$ has a filtration with subquotients isomorphic to direct sums of shifts of
$R_x$ with $x \le v$. Now, using the fact that $\Hom(R_x, R_y) = 0$ if
$x \ne y$  we conclude that $\Hom(R_w, E_v^i) = 0$ (or alternatively
one may use Soergel's Hom formula). The desired
vanishing then follows.

For the second case, assume that  $w < v$ so that $v \not \le
w$. Using \eqref{b} we have to prove that
\[
\Hom_K(F_w, R_v(\ell(v))[m]) = 0.
\]
By a similar argument to that of the previous paragraph we see that
\[\Hom(F_w^i, R_v) = 0 \quad \text{for all $i$.}\]The theorem now follows.
\end{proof}

\end{document}